\documentclass[10pt,bezier]{article}
\usepackage{amsmath,amssymb,amsfonts,graphicx,caption,subcaption}
\usepackage[colorlinks,linkcolor=red,citecolor=blue]{hyperref}

\newtheorem{prethm}{{\bf Theorem}}[section]
\newenvironment{thm}{\begin{prethm}{\hspace{-0.5
em}{\bf.}}}{\end{prethm}}
\newtheorem{prepro}{{\bf Theorem}}

\newtheorem{precor}[prethm]{{\bf Corollary}}
\newenvironment{cor}{\begin{precor}{\hspace{-0.5
em}{\bf.}}}{\end{precor}}
\newtheorem{preconj}[prethm]{{\bf Conjecture}}

\newtheorem{preremark}[prethm]{{\bf Remark}}

\newtheorem{prelem}[prethm]{{\bf Lemma}}
\newenvironment{lem}{\begin{prelem}{\hspace{-0.5
em}{\bf.}}}{\end{prelem}}
\newtheorem{preque}[prethm]{{\bf Question}}

\newtheorem{preobserv}[prethm]{{\bf Observation}}

\newtheorem{preproposition}[prethm]{{\bf Proposition}}

\newtheorem{preproof}{{\bf Proof.}}
\newtheorem{preprooff}{{\bf Proof}}

\newenvironment{proof}[1]{\begin{preproof}{\rm
#1}\hfill{$\Box$}}{\end{preproof}}

\newtheorem{preproofs}{{\bf Second proof of }}

\newtheorem{preprooft}{{\bf Third proof of }}

\newtheorem{preproofF}{{\bf Proofs of}}

\title{\bf\Large 
Bipartite partition-connected factors with small degrees
}
\author{{\normalsize{\sc Morteza Hasanvand${}$} }\vspace{3mm}
\\{\footnotesize{${}$\it Department of Mathematical
 Sciences, Sharif
University of Technology, Tehran, Iran}}
{\footnotesize{}}\\{\footnotesize{ $\mathsf{hasanvand@alum.sharif.edu }$ }}}

\date{}
\def\epsilon {\text{$\varepsilon$}}
\def\OMEGA {\text{$\Theta$}}
\begin{document}
\maketitle
\begin{abstract}{
In this paper, we show that every $2m$-partition-connected graph $G$ has a bipartite $m$-partition-connected factor $H$ such that for each vertex $v$, $d_H(v)\le \lceil \frac{3}{4}d_G(v)\rceil$. A graph $H$ is said to be $m$-partition-connected, if it contains $m$ edge-disjoint spanning trees. As an application, we conclude that tough enough graphs with appropriate number of vertices have a bipartite $m$-partition-connected factor with maximum degree at most $3m+1$. Finally, we prove that tough enough graphs of order at least $3k$ admit a bipartite connected factor whose degrees lie in the set $\{k,2k,3k,4k\}$.
\\
\\
\noindent {\small {\it Keywords}:
\\
Toughness;
bipartite factor;
chromatic number;
connected factor;
partition-connected;
modulo factor.

}} {\small
}
\end{abstract}
%
%
%
%
%
%
%
%
%
%
%
%
%
%
\section{Introduction}
In this article, all graphs have no loop, but multiple edges are allowed.
 Let $G$ be a graph. 
The vertex set and the edge set of $G$ are denoted by $V(G)$ and $E(G)$, respectively. 
The degree $d_G(v)$ of a vertex $v$ is the number of edges of $G$ incident to $v$.
Let $A$ and $B$ be two subsets of $V(G)$.
This pair is said to be {\bf intersecting}, if $A\cap B \neq \emptyset$.
Let $l$ be a real function on subsets of $V(G)$ with $l(\emptyset) =0$.
For notational simplicity, we write $l(G)$ for $l(V(G))$ and write $l(v)$ for $l(\{v\})$.
The function $l$ is said to be {\bf supermodular}, if for all vertex sets $A$ and $B$,
$l(A\cap B)+l(A\cup B)\ge l(A)+l(B).$
Likewise, $l$ is said to be intersecting supermodular, if for all intersecting pairs $A$ and $B$
the above-mentioned inequality holds.
The set function $l$ is called {\bf nonincreasing}, if $l(A)\ge l(B)$, for all nonempty vertex sets
$A$, $B$ with $A\subseteq B$.
Note that several results of this paper can be hold for real functions $l$ such that $\sum_{v\in A}l(v)-l(A)$ is integer for every vertex set $A$.
For clarity of presentation, we will assume that $l$ is integer-valued.
The graph $G$ is said to be {\bf $l$-edge-connected}, if for all nonempty proper vertex sets $A$,
 $d_G(A) \ge l(A)$, where
$d_G(A)$ denotes the number of edges of $G$ with exactly one end in $A$.
Likewise, the graph $G$ is called {\bf $l$-partition-connected}, 
if for every partition $P$ of $V(G)$,
$e_G(P)\ge \sum_{A\in P}l(A)-l(V(G)),$
where $e_G(P)$ denotes the number of edges of $G$ joining different parts of $P$.
It is known that if $l$ is intersecting supermodular, then 
the
vertex set of $G$ can be expressed uniquely as a disjoint union of vertex sets of some
induced $l$-partition-connected subgraphs.
These subgraphs are called the $l$-partition-connected components of $G$.
An $l$-partition-connected graph $G$ is called {\bf minimally $l$-partition-connected}, 
if for every edge $e$ of $G$, the resulting graph $G-e$ is not $l$-partition-connected.
To measure $l$-partition-connectivity of $G$, we define the parameter
$\OMEGA_l(G)=\sum_{A\in P} l(A)-e_G(P),$ where 
 $P$ is the partition of $V(G)$ obtained from $l$-partition-connected components of $G$.
In ~\cite{P}, it was proved that $\OMEGA_l(G)\le \frac{1}{k}\OMEGA_{kl}(G)$, when $k$ is a real number with $k\ge 1$.
A graph is said to be {\bf $m$-tree-connected}, it has $m$ edge-disjoint spanning trees.
By the main result in~\cite{MR0133253, MR0140438}, a graph $G$ is $m$-tree-connected if and only if it is $m$-partition-connected.
Let $t$ be a positive real number, 
a graph $G$ is said to be {\bf $t$-tough}, if $\omega(G\setminus S)\le \max\{1,\frac{1}{t}|S|\}$ for all $S\subseteq V(G)$. 
The {\bf chromatic number} of a graph $G$ is to the minimum number of colors needed to color the vertices of $G$ such that adjacent vertices admit different colors.
 This number is denoted by $\chi(G)$.
For a set of integers $A$, an {\bf $A$-factor} is a spanning subgraph with vertex degrees in $A$.
A {\bf modulo $k$-regular factor} is a spanning subgraph whose degrees are positive and divisible by $k$.
Throughout this article, all variables $k$ and $c$ are positive and integer, unless otherwise stated, and all variables $m$ are nonnegative and integer.
%
%
%
%
%
%
%
%
%
%
%
%

In 1973 Chv\'atal~\cite{MR0316301} conjectured that there exists a positive real number $t_0$ such that every $t_0$-tough graph of order at least three admits a Hamiltonian cycle. In 1989 Win \cite{MR998275} confirmed a weaker version of this conjecture as the following result. In 2000 Ellingham and Zha~\cite{MR1740929} refined this result for graphs with higher toughness by proving that every
$4$-tough graph of order at least three admits a connected $\{2,3\}$-factor.
\begin{thm}{\rm (\cite{MR998275})}\label{thm:Win}
Every
$1$-tough graph admits a connected factor with maximum degree at most $3$.
\end{thm}

Recently, the present author generalized Theorem~\ref{thm:Win} to following tree-connected version for tough enough graphs.
As an application, we concluded that every $4$-tough graph of order at least three admits a connected $\{2,4\}$-factor.
\begin{thm}{\rm (\cite{II})}\label{thm:tough:2m+1}
Every
$m^2$-tough graph of order at least $2m$ admits an $m$-tree-connected factor with maximum degree at most $2m+1$.
\end{thm}

In this paper, we turn our attention to investigate bounded degree tree-connected bipartite factors 
by proving that 
every $2m$-tree-connected graph $G$ has a bipartite $m$-tree-connected factor $H$ such that for each vertex $v$, $d_H(v)\le \lceil \frac{3}{4}d_G(v)\rceil$.
From this result, one can form the following bipartite version for Theorem~\ref{thm:tough:2m+1}.
As an application, we show that every $16$-tough graph of order at least six admit a bipartite connected $\{2,4,6\}$-factor. 
\begin{thm}\label{thm:tough:3m+1}
Every $4m^2$-tough graph of order at least $4m$ admits a bipartite $m$-tree-connected graph with maximum degree at most $3m+1$.
\end{thm}

In 1985 Enomoto, Jackson, Katerinis, and Saito~\cite{MR785651} proved that every $k$-tough graph $G$ with $k|V(G)|$ even admits a $k$-factor. 
In tis paper, we show that tough enough graphs of order at least $3k$ admit a connected bipartite factor whose degrees lie in the set $\{k,2k,3k,4k\}$. More precisely, 
 we apply a combination of the following theorem and Theorem~\ref{thm:tough:3m+1}.
\begin{thm}{\rm (\cite{ModuloBounded})}\label{thm:2k-2}
{Every $(2k-1)$-tree-connected bipartite graph $G$ with $u\in V(G)$ has a connected modulo $k$-regular factor $H$ such that for each 
$v\in V(H)\setminus u$, $ d_H(v) \le d_G(v)- k+1$.
}\end{thm}
%
%
%
%
%
%
%
%
%
%
%
\section{Partition-connected factors with small chromatic number}
In 2008 Thomassen~\cite{MR2501526} showed that
 every $(2\lambda-1)$-edge-connected graph has a bipartite $\lambda$-edge-connected factor.
Recently, the present author refined this result to the following tree-connected version.
\begin{thm}{\rm (\cite{ModuloBounded})}\label{thm:bipartite:tree-connected}
{Every $2m$-tree-connected graph $G$ has a bipartite $m$-tree-connected factor $H$ such that for every vertex set $X$, $d_{H}(X)\ge d_G(X)/2$.
}\end{thm}
In the following theorem, we provide a generalization for Theorem~\ref{thm:bipartite:tree-connected} by giving a sufficient partition-connected condition for the existence of partition-connected factors with bounded chromatic number.
\begin{thm}\label{thm:bipartite:partition-connected}
{Let $G$ be a graph, let $c$ be a positive integer with $c\ge 2$, and let $l$ be a real function on subsets of $V(G)$. 
If $G$ $c\, l / (c-1)$-partition-connected, then it has a $c$-partite $l$-partition-connected factor $H$
 such that for every vertex set $X$, $d_H(X)\ge \frac {c-1}{c}d_G(X).$
}\end{thm}
\begin{proof}
{Let $H$ be an induced $c$-partite factor of $G$ with the maximum $|E(H)|$.
Let $X$ be a vertex subset of $V(G)$
and let $V_1,\ldots, V_c$ be the partite sets of $H$.
For every $i$ with $1\le i \le c$, let $X_i=X\cap V_i$ and define 
 $H_i$ to be the induced $c$-partite factor of $G$ with the partite sets $V_1^i,\ldots, V_{c}^i$, where 
$V_j^i=(V_j \setminus X_j)\cup X_{i+j}$ and $i+j$ is computed modulo $c$.
Since $|E(H)| \ge |E(H_i)|$, we must have 
$$d_H(X)-d_{H_i}(X)=|E(H)|- |E(H_i)|\ge 0.$$
Therefore, 
 $$d_G(X)=\sum_{1\le i\le c}\big(d_{G}(X)-d_{H_i}(X)\big) \ge c \big(d_{G}(X)-d_H(X)\big).$$ 
This implies that $d_H(X)\ge \frac {c-1}{c}d_G(X)$.
 Let $P$ be a partition of $V(G)$. Since $G$ is $c l/ (c-1)$-partition-connected, we must have
$$ \frac{1}{2} \sum_{X\in P}d_H(X)\ge 
\frac{c-1}{c}\sum_{X\in P} \frac{1}{2} d_G(X) = \frac{c-1}{c} e_G(P)\ge 
\frac{c-1}{c}\big (\sum_{X\in P}\frac{c}{c-1}l(X)\;-\frac{c}{c-1}l(G)\big),$$
which implies that 
$$e_H(P) = \frac{1}{2}\sum_{X\in P}d_H(X) \ge \sum_{X\in P}l(X)-l(G).$$
Hence $H$ is $l$-partition-connected, as desired.
}\end{proof}
\begin{cor}
{Every graph $G$ contains a factor $H$ such that for every vertex set $X$,
$d_H(X)\ge \frac{c-1}{c}d_G(X)$, and $\chi(H) \le c$, where $c$ is a given arbitrary integer number with $c\ge 2$.
}\end{cor}
\begin{cor}\label{thm:tree-connected}
{Let $G$ be a graph and let $l$ be a real function on subsets of $V(G)$. 
 If $G$ is $2l$-partition-connected, 
then it has a bipartite $l$-partition-connected factor $H$ such that for every vertex set $X$, $d_H(X)\ge d_G(X)/2$.
}\end{cor}
%
%
%
%
%
%
%
%
%
%
%
%
%
%
%
%
%
\section{Partition-connected factors with small chromatic numbers and degrees}
In this section, we shall give a sufficient condition for a graph to satisfy the assumption of the following theorem.
As an application, we derive a result on partition-connected factors with small chromatic numbers and degrees.
\begin{thm}{\rm (\cite{P})}\label{thm:sufficient}
Let $G_0$ be a graph with the factor $F$ and let $l$ be an intersecting supermodular nonincreasing nonnegative integer-valued function on 
 subsets of $V(G_0)$. 
 Let $\lambda \in [0,1]$ be a real number and let $\eta$ be a real function on $V(G_0)$.
If for all $S\subseteq V(G_0)$, 
$$\OMEGA_l(G_0\setminus S)< 1+\sum_{v\in S}\big(\eta(v)-2l(v)\big)+l(G_0)+l(S)-\lambda(e_G(S)+l(S)),$$
then $G_0$ has an $l$-partition-connected factor $H$ containing $F$ such that for each vertex $v$, 
$d_H(v)\le \lceil \eta(v)-\lambda l(v)\rceil+\max\{0,d_F(v)-l(v)\}.$
\end{thm}
Before stating the main result, let us make the following lemma.
\begin{lem}\label{lem:epsilon}
{Let $G$ be a graph with the spanning subgraph $G_0$, let $l$ be an intersecting supermodular real function on
subsets of $V(G)$,
 and let $\epsilon$ be a positive real number. If $G$ is $l/\epsilon$-partition-connected and $d_{G_0}(X)\ge \epsilon \, d_G(X)$ for every $X\subseteq V(G)$, 
then for every $S\subseteq V(G)$,
$$\OMEGA_l(G_0\setminus S)\le \sum_{v\in S}\big(\frac{1}{2}d_{G_0}(v)+\frac{\epsilon}{2} d_{G}(v)-l(v)\big)+l(G)
-e_{G_0}(S).$$
}\end{lem}
\begin{proof}
{Let $P$ be a partition of $V(G)\setminus S$ obtained from the $l$-partition-connected components of $G\setminus S$.
By the assumption, 
$\sum_{A\in P}\frac{\epsilon}{2}d_{G}(A) \le
 \sum_{A\in P}\frac{1}{2}d_{G_0}(A)=e_{G_0}(P)+\sum_{v\in S}\frac{1}{2}d_{G_0}(v)-e_{G_0}(S).$
Since $G$ is $l/\epsilon $-partition-connected, we must have 
$$\sum_{A\in P}l(A)+\sum_{v\in S}l(v)-l(G)\le
\epsilon \, e_{G}(P\cup \{\{v\}: v\in S\}) =
 \sum_{A\in P}\frac{\epsilon}{2}d_G(A)+\sum_{v\in S}\frac{\epsilon}{2} d_G(v).$$
Therefore, 
$$\OMEGA_l(G_0\setminus S)=\sum_{A\in P}l(A)-e_{G_0}(P)\le \sum_{v\in S}\big(\frac{1}{2}d_{G_0}(v)+\frac{\epsilon}{2} d_{G}(v)-l(v)\big)+l(G)
-e_{G_0}(S).$$
Hence the lemma is proved.
}\end{proof}
Now, we are ready to prove the main result of this section.
\begin{thm}\label{thm:epsilon}
{Let $G$ be a graph, let $l$ be an intersecting supermodular nonincreasing nonnegative integer-valued function on
subsets of $V(G)$,
 and let $k$ and $\epsilon$ be two positive real numbers with $k\ge 1\ge \epsilon $.
Let $G_0$ be a spanning subgraph of $G$ with a factor $F$.
If $G$ is $k l/\epsilon$-partition-connected and $d_{G_0}(X)\ge \epsilon \, d_G(X)$ for every $X\subseteq V(G)$, 
then $G_0$ has an $l$-partition-connected factor $H$ containing $F$ such that for each vertex $v$,
$$d_H(v) \le \lceil\frac{1}{2k}d_{G_0}(v)+\frac{\epsilon}{2k} d_G(v)+\frac{k-1}{k}l(v)\rceil+\max\{0,d_F(v)-l(v)\}.$$
Furthermore, for a given arbitrary vertex $u$, the upper bound can be reduced to
$\lfloor\frac{1}{2k}d_{G_0}(u)+\frac{\epsilon}{2k} d_G(u)+\frac{k-1}{k}(l(u)-l(G))\rfloor+\max\{0,d_F(u)-l(u)\}$.
}\end{thm}
\begin{proof}
{Let $S\subseteq V(G)$. 
By Lemma~\ref{lem:epsilon}, 
$$\OMEGA_{kl}(G_0\setminus S)\le
 \sum_{v\in S}\big(\frac{1}{2}d_{G_0}(v)+\frac{\epsilon}{2} d_{G}(v)-kl(v)\big)+kl(G)
-e_{G_0}(S),$$
which implies that 
$$\OMEGA_{l}(G_0\setminus S)\le 
\frac{1}{k}\OMEGA_{kl}(G_0\setminus S)<1+
 \sum_{v\in S}\big(\eta(v)-2l(v)\big)+l(G)+\frac{k-1}{k}l(S)-\frac{1}{k}e_{G_0}(S),$$
where $\eta(u)= \frac{1}{2k}d_{G_0}(u)+\frac{\epsilon}{2k} d_{G}(u)+l(u)-\frac{k-1}{k}l(G)-q$ and $q$ is a real number with $0\le q<1$ which is sufficiently close to $1$,
and $\eta(v)=\frac{1}{2k}d_{G_0}(v)+\frac{\epsilon}{2k} d_{G}(v)+l(v)$ for all vertices $v$ with $v\neq u$. 
 Now, it is enough to apply Theorem~\ref{thm:sufficient} with $\lambda=1/k$. 
}\end{proof}
For the special case $k=1$, the theorem becomes simpler as the following corollary.
\begin{cor}\label{thm:epsilon}
{Let $G$ be a graph with the spanning subgraph $G_0$, let $l$ be an intersecting supermodular  nonincreasing nonnegative  integer-valued function on
subsets of $V(G)$,
 and let $\epsilon$ be a positive real number.
If $G$ is $ l/\epsilon$-partition-connected and $d_{G_0}(X)\ge \epsilon \, d_G(X)$ for every $X\subseteq V(G)$, 
then $G_0$ has an $l$-partition-connected factor $H$ such that for each vertex $v$,
$$d_H(v) \le \lceil\frac{1}{2}d_{G_0}(v)+\frac{\epsilon}{2} d_G(v)\rceil.$$
Furthermore, for a given arbitrary vertex $u$, the upper bound can be reduced to
$\lfloor d_{G_0}(u)/2+\epsilon\, d_G(u)/2\rfloor$.
}\end{cor}
An application of Corollary~\ref{thm:epsilon} and Theorem~\ref{thm:bipartite:partition-connected} is given in the next result.
\begin{cor}\label{cor:2m:3/4}
{Every $mc /(c-1)$-partition-connected graph $G$ with $c\ge 2$ has a $c$-partite $m$-tree-connected factor $H$ such that for each vertex $v$,
$$d_H(v)\le\big \lceil \frac{2c-1}{2c}d_G(v)\big\rceil.$$
Furthermore, for an arbitrary given vertex $u$, the upper bound can be reduced to $\lfloor \frac{2c-1}{2c}d_G(u)\rfloor$.
}\end{cor}
\begin{proof}
{By Theorem~\ref{thm:bipartite:partition-connected}, 
the graph $G$ has a $c$-partite $m$-tree-connected factor $G_0$ such that for every vertex set $X$,
 $d_{G_0}(X)\ge (c-1)d_G(X)/c$.
By applying Theorem~\ref{thm:epsilon} with $\epsilon =(c-1)/c$, the graph $G_0$ has an $m$-tree-connected factor $H$ such that for 
each vertex $v$, $$d_H(v) \le \lceil\frac{1}{2}d_{G_0}(v)
+ \frac{c-1}{2c}d_G(v)\rceil \le \lceil \frac{2c-1}{2c}d_G(v)\rceil.$$
 Hence the proof is completed.
}\end{proof}
\begin{cor}\label{cor:2m:3/4}
{Every $2m$-tree-connected graph $G$ has a bipartite $m$-tree-connected factor $H$ such that for each vertex $v$,
$$d_H(v)\le \lceil \frac{3}{4}d_G(v)\rceil.$$
Furthermore, for an arbitrary given vertex $u$, the upper bound can be reduced to $\lfloor \frac{3}{4}d_G(u)\rfloor.$
}\end{cor}
The next theorem gives another sufficient condition for a graph to satisfy the assumption of Theorem~\ref{thm:sufficient}.
\begin{thm}
Let $G$ be a graph with the spanning subgraph $G_0$, let $k$ be a real number with $k\ge 1$, and let $l$ be an intersecting supermodular real function on 
 subsets of $V(G)$. 
Assume that for every vertex set $A$, $d_{G_0}(A)\ge \frac{1}{k}d_G(A)$.
If $S\subseteq V(G)$, then 
$$\OMEGA_l(G_0\setminus S)\le \frac{1}{k}\OMEGA_{kl}(G\setminus S)
+\sum_{v\in S}\big(\frac{1}{2}d_{G_0}(v)-\frac{1}{2k}d_G(v)\big)-e_{G_0}(S)+\frac{1}{k}e_{G}(S).$$
\end{thm}
\begin{proof}
{Let $P$ be a partition of $V(G)\setminus S$ obtained from the $l$-partition-connected components of $G\setminus S$.
By the assumption, 
$$\frac{1}{k}e_{G}(P)+\sum_{v\in S}\frac{1}{2k}d_{G}(v)-\frac{1}{k}e_{G}(S)=
\sum_{A\in P}\frac{1}{2k}d_{G}(A) \le
 \sum_{A\in P}\frac{1}{2}d_{G_0}(A)=e_{G_0}(P)+\sum_{v\in S}\frac{1}{2}d_{G_0}(v)-e_{G_0}(S).$$
Therefore, 
$$\OMEGA_l(G_0\setminus S)=\sum_{A\in P}l(A)-e_{G_0}(P)\le 
\frac{1}{k}(\sum_{A\in P}kl(A)\,-e_{G}(P))
+\sum_{v\in S}\big(\frac{1}{2}d_{G_0}(v)-\frac{1}{2k}d_G(v)\big)-e_{G_0}(S)+\frac{1}{k}e_{G}(S).$$
This inequality can complete the proof.
}\end{proof}
%
%
%
%
%
%
%
%
%
%
%
\section{Bipartite connected modulo regular factors with small degrees}
In 2014 Thomassen~\cite{MR3194201} showed that graphs with high enough edge-connectivity admit
 a bipartite modulo $k$-regular factor. This result is recently developed to the following connected factor version.
\begin{thm}{\rm (\cite{ModuloBounded})}
{Every $(4k-2)$-tree-connected graph admits a bipartite connected modulo $k$-regular factor.
}\end{thm}
In the following, we shall give a sufficient toughness condition for the existence of a bipartite connected modulo regular factor with small degrees.
\begin{thm}\label{thm:application:toughenough}
{Every $16k^2$-tough graph of order at least $3k$ admits a bipartite connected $\{k,2k,3k,4k\}$-factor.
}\end{thm}
\begin{proof}
{If $|V(G)| < 8k-4$, the graph $G$ must be complete and the proof is straightforward.
So, suppose $|V(G)| \ge 8k-4$.
By Theorem~\ref{thm:tough:3m+1}, the graph $G$ has a bipartite $(2k-1)$-tree-connected factor $H$ such that 
$\Delta(H)\le 3(2k-1)+1=6k-2$.
We may asume that $H$ is minimally $(2k-1)$-tree-connected and so it has a vertex $u$ with $d_{H}(u)\le 2(2k-1)$.
By Theorem~\ref{thm:2k-2}, the graph $H$ has a connected factor $F$
 with degrees divisible by $k$ such that for each $v\in V(F)\setminus u$, 
$$0< d_F(v) \le d_{H}(v)-k+1 \le 6k-2-k+1<5k.$$
In addition, for the special vertex $u$, we must automatically have $d_F(u) \le 4k$.
Hence $F$ is the desired factor we are looking for.
}\end{proof}
A special case of Theorem~\ref{thm:application:toughenough} can be refined to the following stronger version.
\begin{thm}
Every $16$-tough graph of order at least six admits a bipartite connected $\{2,4,6\}$-factor.
\end{thm}
\begin{proof}
{If $|V(G)| <16$, the graph $G$ must be complete and the proof is straightforward. 
If $|V(G)| \ge 16$, then by Theorem~\ref{thm:tough:3m+1}, 
the graph $G$ has a bipartite $2$-tree-connected factor $H$ such that $\Delta(H)\le 7$.
Now, it is enough to consider a spanning Eulerian subgraph of $H$ \cite{MR519177}.
}\end{proof}
%
%
%
%
%
%
%
%
%
%
%
%
%

\end{document}